\documentclass[11pt]{amsart}
\usepackage{amsmath}
\usepackage{amscd}
\usepackage{amsthm}
\usepackage{graphicx}
\usepackage{amssymb}
\usepackage{epstopdf}
\theoremstyle{plain}
\newtheorem{Thm}{Theorem}[section]

\newtheorem{Prop}[Thm]{Proposition}
\newtheorem{Cor}[Thm]{Corollary}
\newtheorem{Lem}[Thm]{Lemma}

\theoremstyle{definition}

\newtheorem{Rem}[Thm]{Remark}
 
\numberwithin{equation}{section}

\title{On the derived category of a weighted projective threefold}
\author{Yujiro Kawamata \\ \\
{\it dedicated to Professor Fabrizio Catanese for his seventieth birthday}}
\usepackage{fancyhdr}
\begin{document}
\maketitle

\begin{abstract}
We calculate a semi-orthogonal decomposition of the bounded derived category of coherent sheaves on 
$\mathbf{P}(1,1,1,3)$ using a tilting bundle.
\end{abstract}



\section{Introduction}

Let us consider a bounded derived category of coherent sheaves $D^b(S)$ for a weighted projective surface 
$S = \mathbf{P}(1,1,2)$.
It  is generated by coherent sheaves $\mathcal{O}_S(-2), \mathcal{O}_S(-1), \mathcal{O}_S$, where 
$\mathcal{O}_S(-1)$ is a reflexive sheaf of rank $1$ which is not invertible.
But there is a locally free sheaf $F_S$ of rank $2$ on $S$ sitting in an exact sequence
\[
0 \to \mathcal{O}_S(-1) \to F_S \to \mathcal{O}_S(-1) \to 0
\]
such that $\mathcal{O}_S(-2),F_S,\mathcal{O}_S$ generate $D^b(S)$, and such that 
$\text{End}(F_S) \cong k[t]/(t^2)$ and $\text{Ext}^i(F_S,F_S) \cong 0$ for $i > 0$.
In this way we obtain a semi-orthogonal decomposition
\[
D^b(S) = \langle \mathcal{O}_S(-2),F_S,\mathcal{O}_S \rangle \cong \langle D^b(k), D^b(k[t]/(t^2)), D^b(k) \rangle.
\]
This kind of phenomena is greatly generalized in \cite{KKS} to surfaces with cyclic quotient singularities
and to $3$-dimensional varieties in \cite{KPS}.
But the generalizations in dimension $3$ are mostly concerned only with the case of varieties having hypersurface singularities, 
especially ordinary double points.
In this article we would like to consider higher dimensional varieties with toric singularities taking an example of 
a weighted projective space.

The purpose of this short article is to provide an example which shows that a similar construction of a semi-orthogonal decomposition
of a bounded derived category is possible for a higher dimensional variety with a cyclic quotient singularity.
This work is inspired by talks by Professors Martin Kalck and Evgeny Shinder in a Zoom conference
\lq\lq Categories and Birational Geometry'' held in December 2020 
(http://www.mathnet.ru/php/conference.\break phtml?confid=1835\&option\_lang=eng).
The author would like to thank them for the comments on the first version of the article.

\begin{Thm}
Let $X = \mathbf{P}(1,1,1,3)$ be a weighted projective space of dimension $3$.
Then there exist locally free sheaves $F$ and $G$ of ranks $3$ and $6$, respectively, 
which satisfy the following conditions:

(1) $\text{Ext}^i(F \oplus G,F \oplus G) \cong 0$ for $i > 0$.

(2) There is a semi-orthogonal decomposition
\[
D^b(X) = \langle F \oplus G,\mathcal{O}_X, \mathcal{O}_X(3) \rangle \cong \langle D^b(R_X), D^b(k), D^b(k) \rangle
\]
where $R_X = \text{End}(F \oplus G)$ with $\dim R_X = 45$.
\end{Thm}

\section{Preliminaries}

We use the following notation.
$D^b(X)$ (resp. $D^-(X)$) denotes the bounded (resp. bounded from above) 
derived category of coherent sheaves on a variety $X$.
$D^b(R)$ (resp. $D^-(R)$) denotes the bounded (resp. bounded from above) 
derived category of finitely generated right $R$-modules for an associative $k$-algebra $R$.
We write
\[
\text{Hom}^*(A,B) := H^*(\mathbf{R}\text{Hom}(A,B)) = \bigoplus_i \text{Hom}(A,B[i])[-i].
\]
We work over $k = \mathbf{C}$.

\vskip 1pc

We recall some definitions from \cite{Bondal}.
Let $A$ be a $k$-linear triangulated category.
A set of objects $\{a_i\} \subset A$ is said to {\em generate} $A$ when the following holds: for $a \in A$,
if $\text{Hom}(a_i,a[j]) = 0$ for all $i$ and $j$, then $a \cong 0$.

We say that $A$ has a {\em semi-orthogonal decomposition} to triangulated full subcategories $B_1,\dots,B_n$ 
and denote
\[
A = \langle B_1,\dots,B_n \rangle
\]
if (1) $\text{Hom}(b_i,b_j) = 0$ for all $b_i \in B_i$ and $b_j \in B_j$ such that $i > j$, and
(2) $A$ is the smallest triangulated subcategory which contains all $B_i$. 
When $B_i$ is generated by $b_i$ for each $i$, then we also write
\[
A = \langle b_1,\dots,b_n \rangle.
\]

Let $i_*: B \to A$ be a triangulated full subcategory. 
$B$ is said to be {\em left (resp. right) admissible} if $i_*$ has a left adjoint $i^*$ (resp. a right adjoint $i^!$).
The {\em left (resp. right) semi-orthogonal complement} ${}^{\perp}B$ (resp. $B^{\perp}$) is  
defined as a full subcategory of $A$ such that 
\[
\begin{split}
&{}^{\perp}B = \{a \in A \mid \text{Hom}(a,b) = 0\,\,\, \forall b \in B\}, \\
&B^{\perp} = \{a \in A \mid \text{Hom}(b,a) = 0\,\,\, \forall b \in B\}.
\end{split}
\]
If $B$ is a left (resp. right) admissible subcategory, then 
we have a semi-orthogonal decomposition
\[
\begin{split}
A = \langle B, {}^{\perp}B \rangle \quad (\text{resp. }= \langle B^{\perp}, B \rangle).
\end{split}
\]
Indeed we have the following distinguished triangles for any $a \in A$:
\[
\begin{split}
&c \to a \to i_*i^*a \to c[1], \\ 
&d[-1] \to i_*i^!a \to a \to d,
\end{split}
\]
such that we have $i_*i^*a \in B$, $c \in {}^{\perp}B$, $i_*i^!a \in B$ and $d \in B^{\perp}$.

\vskip 1pc

We recall a version of the tilting theory in \cite{TU}.
Let $X$ be a projective variety and let $P$ be a {\em perfect complex}, an object in $D^b(X)$ which is locally isomorphic to a
bounded complex of locally free sheaves.
$P$ is said to be {\em tilting} if $\text{Hom}(P,P[i]) = 0$ for $i \ne 0$.
The endomorphism algebra $R = \text{End}(P)$ is an associative algebra.

\begin{Lem}[\cite{TU}~Lemma 3.3]
Let $P \in D^b(X)$ be a tilting object, and let $R = \text{End}(P)$.
Let $\Phi: D^-(X) \to D^-(R)$ and $\Psi: D^-(R) \to D^-(X)$ be functors defined by 
$\Phi(\bullet) = \mathbf{R}\text{Hom}(P,\bullet)$
and $\Psi(\bullet) = \bullet \otimes^{\mathbf{L}}_R P$.
Let $A^-$ be the essential image of the functor $\Psi$, and let $A^b = A^- \cap D^b(X)$.
Then $\Psi$ is a left adjoint functor of $\Phi$, and 
$\Phi$ induces an equivalence of triangulated categories $A^- \cong D^-(R)$.
Moreover, $\Phi(D^b(X)) = D^b(R)$, and $\Phi$ induces an equivalence $A^b \cong D^b(R)$. 
\end{Lem}

\begin{proof}
\cite{TU}~Lemma 3.3 treats the case that $A^- = D^-(X)$.
But the same proof works in our situation.
\end{proof}

\begin{Cor}
Let $P_1, \dots, P_n \in D^b(X)$ be tilting objects, and let $R_i = \text{End}(P_i)$.
Assume that $\{P_i\}$ generates $D^-(X)$ and that $\text{Hom}^*(P_i,P_j) = 0$ for $i > j$. 
Then there is a semi-orthogonal decomposition
\[
D^b(X) = \langle P_1,\dots,P_n \rangle \cong \langle D^b(R_1), \dots, D^b(R_n) \rangle.
\]
\end{Cor}

\begin{proof}
We have functors $\Phi_i: D^-(X) \to D^-(R_i)$ and 
$\Psi_i: D^-(R_i) \to D^-(X)$ as in the lemma.
Let $A^-_i$ be the essential image of the functor $\Psi_i$, and let $A^b_i = A^-_i \cap D^b(X)$.
We have $A^b_i \cong D^b(R_i)$ by $\Phi_i$.

$\Phi_n$ is a right adjoint functor of $\Psi_n$, hence $A^-_n$ is a right admissible subcategory.
Thus we have a semi-orthogonal decomposition
\[
D^-(X) = \langle D^-_n, A^-_n \rangle
\]
with $D^-_n = (A^-_n)^{\perp}$.
$\Phi_{n-1}$ induces a right adjoint functor $D^-_n \to D^-(R_{n-1})$ of the functor 
$\Psi_{n-1}: D^-(R_{n-1}) \to A^-_{n-1} \subset D^-_n$.
Hence we have a semi-orthogonal decomposition
\[
D^-_n = \langle D^-_{n-1}, A^-_{n-1} \rangle 
\]
with $D^-_{n-1} = (A^-_{n-1})^{\perp}$.
In this way, we obtain a 
semi-orthogonal decomposition
\[
D^-(X) = \langle A^-_1, \dots, A^-_n \rangle. 
\]
By taking the intersection with $D^b(X)$, we obtain our claim using the last part of the lemma.
\end{proof}

\section{Proof}

Let $X = \mathbf{P}(1,1,1,3)$ be a weighted projective space of dimension $3$.
It is a projective cone over $\mathbf{P}^2$ by $\mathcal{O}_{\mathbf{P}^2}(3)$.
$X$ has a Gorenstein isolated quotient singularity of type $\frac 13(1,1,1)$ at $P = [0:0:0:1]$.
We have $K_X \cong \mathcal{O}_X(-6)$.
Let $Z \cong \mathbf{P}^2$ be a hyperplane section of $X$ at infinity.  
We have $Z \in \vert \mathcal{O}_X(3) \vert$, and 
$N_{X/Z} = \mathcal{O}_Z(3)$.

Let $F$ and $G$ be locally free sheaves of ranks $3$ and $6$ on $X$ defined by natural exact sequences:
\[
\begin{split}
&0 \to F \to \mathcal{O}_X^3 \to \mathcal{O}_Z(1) \to 0, \\
&0 \to G \to \mathcal{O}_X^6 \to \mathcal{O}_Z(2) \to 0.
\end{split}
\]

\begin{Lem}
$F,G,\mathcal{O}_X,\mathcal{O}_X(3)$ generate $D^-(X)$.
\end{Lem}

\begin{proof}
Let $D$ be the smallest full subcategory of $D^b(X)$ which contains $F,G, \mathcal{O}_X,\mathcal{O}_X(3)$ 
and closed under shifts and cone constructions.
Then $\mathcal{O}_Z(1),\mathcal{O}_Z(2),\mathcal{O}_Z(3)$ are contained in $D$.
Hence $D^b(Z) \subset D$.
It follows that $\mathcal{O}_X(3m) \in D$ for all $m$.

For any non-zero object $a \in D^-(X)$, if $i$ is the largest integer such that $H^i(a) \ne 0$, then 
there is a non-zero morphism $\mathcal{O}_X(3m)[-i] \to a$ for some integer $m$.
Therefore $D^-(X)$ is generated by $F,G, \mathcal{O}_X,\mathcal{O}_X(3)$.
\end{proof}

\begin{Lem}
(1) $\text{Hom}^*(\mathcal{O}_X,\mathcal{O}_X) \cong k$.

(2) $\text{Hom}^*(\mathcal{O}_X,\mathcal{O}_Z(1)) \cong k^3$.

(3) $\text{Hom}^*(\mathcal{O}_X,\mathcal{O}_Z(2)) \cong k^6$.

(4) $\text{Hom}^*(\mathcal{O}_Z(1),\mathcal{O}_X) \cong k^6[-1]$.

(5) $\text{Hom}^*(\mathcal{O}_Z(2),\mathcal{O}_X) \cong k^3[-1]$.

(6) $\text{Hom}^*(\mathcal{O}_Z(1),\mathcal{O}_Z(1)) \cong k \oplus k^{10}[-1]$.

(7) $\text{Hom}^*(\mathcal{O}_Z(1),\mathcal{O}_Z(2)) \cong k^3 \oplus k^{15}[-1]$.

(8) $\text{Hom}^*(\mathcal{O}_Z(2),\mathcal{O}_Z(1)) \cong k^6[-1]$.
\end{Lem}

\begin{proof}
(1), (2), (3) are obvious.

(4) $\text{Hom}^i(\mathcal{O}_Z(1),\mathcal{O}_X) \cong \text{Hom}^{3-i}(\mathcal{O}_X,\mathcal{O}_Z(-5))^* 
\cong H^{3-i}(Z,\mathcal{O}_Z(-5))^* \cong H^{2-(3-i)}(Z,\mathcal{O}_Z(2))$.

(5) $\text{Hom}^i(\mathcal{O}_Z(2),\mathcal{O}_X) \cong \text{Hom}^{3-i}(\mathcal{O}_X,\mathcal{O}_Z(-4))^* 
\cong H^{3-i}(Z,\mathcal{O}_Z(-4))^* \cong H^{2-(3-i)}(Z,\mathcal{O}_Z(1))$.

(6) $\mathcal{H}om^*(\mathcal{O}_Z(1),\mathcal{O}_Z(1)) \cong \mathcal{O}_Z \oplus N_{Z/X}[-1] 
\cong \mathcal{O}_Z \oplus \mathcal{O}_Z(3)[-1]$, and 
$\text{Hom}^*(\mathcal{O}_Z(1),\mathcal{O}_Z(1)) \cong H^0(Z,\mathcal{H}om^*(\mathcal{O}_Z(1),\mathcal{O}_Z(1))) 
\cong k \oplus k^{10}[-1]$.

(7), (8) as well as (6) follow from an exact sequence $0 \to \mathcal{O}_X(-2) \to \mathcal{O}_X(1) \to \mathcal{O}_Z(1) \to 0$ with 
$\text{Hom}^*(\mathcal{O}_X(1), \mathcal{O}_Z(1)) \cong k$, 
$\text{Hom}^*(\mathcal{O}_X(-2), \mathcal{O}_Z(1)) \cong k^{10}$, 
$\text{Hom}^*(\mathcal{O}_X(1), \mathcal{O}_Z(2)) \cong k^3$, 
$\text{Hom}^*(\mathcal{O}_X(-2), \mathcal{O}_Z(2)) \cong k^{15}$, 
$\text{Hom}^*(\mathcal{O}_X(1), \mathcal{O}_Z) \cong 0$, 
$\text{Hom}^*(\mathcal{O}_X(-2), \mathcal{O}_Z) \cong k^6$.
\end{proof}

\begin{Lem}
(1) $\text{Hom}^*(\mathcal{O}_X,F) \cong 0$.

(2) $\text{Hom}^*(\mathcal{O}_X,G) \cong 0$.

(3) $\text{Hom}^*(\mathcal{O}_X(3),F) \cong 0$.

(4) $\text{Hom}^*(\mathcal{O}_X(3),G) \cong 0$.
\end{Lem}

\begin{proof}
(1) and (2) follow from $\text{Hom}^*(\mathcal{O}_X,\mathcal{O}_X^3) \cong \text{Hom}^*(\mathcal{O}_X,\mathcal{O}_Z(1))$ and 
$\text{Hom}^*(\mathcal{O}_X,\mathcal{O}_X^6) \cong \text{Hom}^*(\mathcal{O}_X,\mathcal{O}_Z(2))$.

(3) and (4) follow from 
$\text{Hom}^*(\mathcal{O}_X(3),\mathcal{O}_X) \cong \text{Hom}^*(\mathcal{O}_X(3),\mathcal{O}_Z(1)) 
\cong \text{Hom}^*(\mathcal{O}_X(3),\mathcal{O}_Z(2)) \cong 0$.
\end{proof}

\begin{Lem}
(1) $\text{Hom}^*(F,\mathcal{O}_X) \cong k^9$.

(2) $\text{Hom}^*(F,\mathcal{O}_Z(1)) \cong k^{18}$.

(3) $\text{Hom}^*(F,\mathcal{O}_Z(2)) \cong k^{30}$.

(4) $\text{Hom}^*(G,\mathcal{O}_X) \cong k^9$.

(5) $\text{Hom}^*(G,\mathcal{O}_Z(1)) \cong k^{24}$.

(6) $\text{Hom}^*(G,\mathcal{O}_Z(2)) \cong k^{45}$.
\end{Lem}

\begin{proof}
(1) follows from $\text{Hom}^*(\mathcal{O}_Z(1),\mathcal{O}_X) \cong k^6[-1]$ and 
$\text{Hom}^*(\mathcal{O}_X^3,\mathcal{O}_X) \cong k^3$ with an exact sequence
\[
0 \to \text{Hom}(\mathcal{O}_X^3,\mathcal{O}_X) \to \text{Hom}(F,\mathcal{O}_X) \to 
\text{Ext}^1(\mathcal{O}_Z(1),\mathcal{O}_X) 
\to 0.
\]

(2) follows from $\text{Hom}^*(\mathcal{O}_Z(1),\mathcal{O}_Z(1)) \cong k \oplus k^{10}[-1]$ and 
$\text{Hom}^*(\mathcal{O}_X^3,\mathcal{O}_Z(1)) \cong k^9$ with an exact sequence
\[
0 \to \text{Hom}(\mathcal{O}_Z(1),\mathcal{O}_Z(1)) \to \text{Hom}(\mathcal{O}_X^3,\mathcal{O}_Z(1)) 
\to \text{Hom}(F,\mathcal{O}_Z(1)) \to \text{Ext}^1(\mathcal{O}_Z(1),\mathcal{O}_Z(1)) \to 0.
\]

(3) follows from $\text{Hom}^*(\mathcal{O}_Z(1),\mathcal{O}_Z(2)) \cong k^3 \oplus k^{15}[-1]$ and 
$\text{Hom}^*(\mathcal{O}_X^3,\mathcal{O}_Z(2)) \cong k^{18}$ with an exact sequence
\[
0 \to \text{Hom}(\mathcal{O}_Z(1),\mathcal{O}_Z(2)) \to \text{Hom}(\mathcal{O}_X^3,\mathcal{O}_Z(2)) 
\to \text{Hom}(F,\mathcal{O}_Z(2)) \to \text{Ext}^1(\mathcal{O}_Z(1),\mathcal{O}_Z(2)) \to 0.
\]

(4) follows from $\text{Hom}^*(\mathcal{O}_Z(2),\mathcal{O}_X) \cong k^3[-1]$ and 
$\text{Hom}^*(\mathcal{O}_X^6,\mathcal{O}_X) \cong k^6$ with an exact sequence
\[
0 \to \text{Hom}(\mathcal{O}_X^6,\mathcal{O}_X) \to \text{Hom}(G,\mathcal{O}_X) \to 
\text{Ext}^1(\mathcal{O}_Z(2),\mathcal{O}_X) 
\to 0.
\]

(5) follows from $\text{Hom}^*(\mathcal{O}_Z(2),\mathcal{O}_Z(1)) \cong k^6[-1]$ and 
$\text{Hom}^*(\mathcal{O}_X^6,\mathcal{O}_Z(1)) \cong k^{18}$ with an exact sequence
\[
0 \to \text{Hom}(\mathcal{O}_X^6,\mathcal{O}_Z(1)) \to \text{Hom}(G,\mathcal{O}_Z(1)) \to 
\text{Ext}^1(\mathcal{O}_Z(2),\mathcal{O}_Z(1)) \to 0.
\]

(6) follows from $\text{Hom}^*(\mathcal{O}_Z(2),\mathcal{O}_Z(2)) \cong k \oplus k^{10}[-1]$ and 
$\text{Hom}^*(\mathcal{O}_X^6,\mathcal{O}_Z(2)) \cong k^{36}$ with an exact sequence
\[
0 \to \text{Hom}(\mathcal{O}_Z(2),\mathcal{O}_Z(2)) \to \text{Hom}(\mathcal{O}_X^6,\mathcal{O}_Z(2)) \to 
\text{Hom}(G,\mathcal{O}_Z(2)) \to \text{Ext}^1(\mathcal{O}_Z(2),\mathcal{O}_Z(2)) \to 0.
\]
\end{proof}

\begin{Prop}
(1) $\text{Hom}^*(F,F) \cong k^9$.

(2) $\text{Hom}^*(G,G) \cong k^9$.

(3) $\text{Hom}^*(F,G) \cong k^{24}$.

(4) $\text{Hom}^*(G,F) \cong k^3$.
\end{Prop}

\begin{proof}
(1) It is sufficient to prove that the natural homomorphism 
$\text{Hom}(F,\mathcal{O}_X^3) \to \text{Hom}(F,\mathcal{O}_Z(1))$ 
is surjective.
We have a commutative diagram:
\[
\begin{CD}
\text{Hom}(\mathcal{O}_X^3,\mathcal{O}_X^3) @>>> \text{Hom}(F,\mathcal{O}_X^3) @>>> 
\text{Hom}^1(\mathcal{O}_Z(1),\mathcal{O}_X^3) @>>> 0 \\
@V{\cong}VV @VVV @VVV \\
\text{Hom}(\mathcal{O}_X^3,\mathcal{O}_Z(1)) @>>> \text{Hom}(F,\mathcal{O}_Z(1)) @>>> 
\text{Hom}^1(\mathcal{O}_Z(1), \mathcal{O}_Z(1)) @>>> 0
\end{CD}
\]
Since the left vertical arrow is bijective, it is sufficient to prove that the right vertical arrow is surjective.
We have another commutative diagram:
\[
\begin{CD}
\text{Hom}(\mathcal{O}_X(-2),\mathcal{O}_X^3) @>{\cong}>> \text{Hom}^1(\mathcal{O}_Z(1),\mathcal{O}_X^3) \\
@VVV @VVV \\
\text{Hom}(\mathcal{O}_X(-2), \mathcal{O}_Z(1)) @>>> \text{Hom}^1(\mathcal{O}_Z(1),\mathcal{O}_Z(1)) @>>> 0
\end{CD}
\]
The left vertical arrow is surjective, hence we have our claim.

(2) It is sufficient to prove that the natural homomorphism 
$\text{Hom}(G,\mathcal{O}_X^6) \to \text{Hom}(G,\mathcal{O}_Z(2))$ 
is surjective.
We have a commutative diagram:
\[
\begin{CD}
\text{Hom}(\mathcal{O}_X^6,\mathcal{O}_X^6) @>>> \text{Hom}(G,\mathcal{O}_X^6) @>>> 
\text{Hom}^1(\mathcal{O}_Z(2),\mathcal{O}_X^6) 
@>>> 0 \\
@V{\cong}VV @VVV @VVV \\
\text{Hom}(\mathcal{O}_X^6,\mathcal{O}_Z(2)) @>>> \text{Hom}(G,\mathcal{O}_Z(2)) @>>> 
\text{Hom}^1(\mathcal{O}_Z(2), \mathcal{O}_Z(2)) 
@>>> 0
\end{CD}
\]
Since the left vertical arrow is bijective, it is sufficient to prove that the right vertical arrow is surjective.
We have another commutative diagram:
\[
\begin{CD}
\text{Hom}(\mathcal{O}_X(-1),\mathcal{O}_X^6) @>{\cong}>> \text{Hom}^1(\mathcal{O}_Z(2),\mathcal{O}_X^6) \\
@VVV @VVV \\
\text{Hom}(\mathcal{O}_X(-1), \mathcal{O}_Z(2)) @>>> \text{Hom}^1(\mathcal{O}_Z(2),\mathcal{O}_Z(2)) @>>> 0
\end{CD}
\]
The left vertical arrow is surjective, hence we have our claim.

(3) It is sufficient to prove that the natural homomorphism 
$\text{Hom}(F,\mathcal{O}_X^6) \to \text{Hom}(F,\mathcal{O}_Z(2))$ 
is surjective.
We have a commutative diagram:
\[
\begin{CD}
\text{Hom}(\mathcal{O}_X^3,\mathcal{O}_X^6) @>>> \text{Hom}(F,\mathcal{O}_X^6) @>>> 
\text{Hom}^1(\mathcal{O}_Z(1),\mathcal{O}_X^6) @>>> 0 \\
@V{\cong}VV @VVV @VVV \\
\text{Hom}(\mathcal{O}_X^3,\mathcal{O}_Z(2)) @>>> \text{Hom}(F,\mathcal{O}_Z(2)) @>>> 
\text{Hom}^1(\mathcal{O}_Z(1), \mathcal{O}_Z(2)) @>>> 0
\end{CD}
\]
Since the left vertical arrow is bijective, it is sufficient to prove that the right vertical arrow is surjective.
We have another commutative diagram:
\[
\begin{CD}
\text{Hom}(\mathcal{O}_X(-2),\mathcal{O}_X^6) @>{\cong}>> \text{Hom}^1(\mathcal{O}_Z(1),\mathcal{O}_X^6) \\
@VVV @VVV \\
\text{Hom}(\mathcal{O}_X(-2), \mathcal{O}_Z(2)) @>>> \text{Hom}^1(\mathcal{O}_Z(1),\mathcal{O}_Z(2)) @>>> 0
\end{CD}
\]
The left vertical arrow is surjective, hence we have our claim.

(4) It is sufficient to prove that the natural homomorphism 
$\text{Hom}(G,\mathcal{O}_X^3) \to \text{Hom}(G,\mathcal{O}_Z(1))$ 
is surjective.
We have a commutative diagram:
\[
\begin{CD}
\text{Hom}(\mathcal{O}_X^6,\mathcal{O}_X^3) @>>> \text{Hom}(G,\mathcal{O}_X^3) @>>> 
\text{Hom}^1(\mathcal{O}_Z(2),\mathcal{O}_X^3) 
@>>> 0 \\
@V{\cong}VV @VVV @VVV \\
\text{Hom}(\mathcal{O}_X^6,\mathcal{O}_Z(1)) @>>> \text{Hom}(G,\mathcal{O}_Z(1)) @>>> 
\text{Hom}^1(\mathcal{O}_Z(2), \mathcal{O}_Z(1)) 
@>>> 0
\end{CD}
\]
Since the left vertical arrow is bijective, it is sufficient to prove that the right vertical arrow is surjective.
We have another commutative diagram:
\[
\begin{CD}
\text{Hom}(\mathcal{O}_X(-1),\mathcal{O}_X^6) @>{\cong}>> \text{Hom}^1(\mathcal{O}_Z(2),\mathcal{O}_X^6) \\
@VVV @VVV \\
\text{Hom}(\mathcal{O}_X(-1), \mathcal{O}_Z(2)) @>>> \text{Hom}^1(\mathcal{O}_Z(2),\mathcal{O}_Z(2)) 
@>>> 0
\end{CD}
\]
The left vertical arrow is surjective, hence we have our claim.
\end{proof}

\begin{Rem} 
(1) There is a smooth Deligne-Mumford stack $\tilde X$ with a projection (a \lq\lq non-commutative crepant resolution'')
$\pi: \tilde X \to X = \mathbf{P}(1,1,1,3)$ which is an isomorphism over the smooth locus of $X$.
By \cite{stack} \S5, $\tilde X$ has a full exceptional collection of length $6$ consisting of invertible sheaves: 
\[
D^b(\tilde X) = \langle \mathcal{O}_{\tilde X}, \mathcal{O}_{\tilde X}(1), \mathcal{O}_{\tilde X}(2), 
\mathcal{O}_{\tilde X}(3), \mathcal{O}_{\tilde X}(4), \mathcal{O}_{\tilde X}(5) \rangle.
\]
\vskip 1pc

(2) The sheaf of $\mathcal{O}_X$-algebras $\mathcal{E}nd(F)$ has fibers isomorphic to the matrix algebra 
$M(3,k)$, 
but it seems that it is not isomorphic to $M(3,\mathcal{O}_X)$. 
We have $45 = 3^2 + 6^2$.
But we do not know whether there are simpler tilting bundles.

\vskip 1pc

(3) The locally free sheaf $F_S$ in \S1 satisfies the following commutative diagram:
\[
\begin{CD}
@. @. 0 @. 0 \\
@. @. @VVV @VVV \\
0 @>>> \mathcal{O}_S(-1) @>>> F_S @>>> \mathcal{O}_S(-1) @>>> 0 \\
@. @V=VV @VVV @VVV \\
0 @>>> \mathcal{O}_S(-1) @>>> \mathcal{O}_S^2 @>>> \mathcal{O}_S(1) @>>> 0 \\
@. @. @VVV @VVV \\
@. @. \mathcal{O}_C(1) @>=>> \mathcal{O}_C(1) \\
@. @. @VVV @VVV \\
@. @. 0 @. 0 
\end{CD}
\]
where $\mathbf{P}^1 \cong C \subset S$ is a curve at infinity.
Similarly $F$ and $G$ satisfy the following:
\[
\begin{CD}
@. @. @. 0 @. 0 \\
@. @. @. @VVV @VVV \\
0 @>>> \mathcal{O}_X(-2) @>>> \mathcal{O}_X(-1)^3 @>>> F @>>> \mathcal{O}_X(-2) @>>> 0 \\
@. @V=VV @V=VV @VVV @VVV \\
0 @>>> \mathcal{O}_X(-2) @>>> \mathcal{O}_X(-1)^3 @>>> \mathcal{O}_X^3 @>>> \mathcal{O}_X(1) 
@>>> 0 \\
@. @. @. @VVV @VVV \\
@. @. @. \mathcal{O}_Z(1) @>=>> \mathcal{O}_Z(1) \\
@. @. @. @VVV @VVV \\
@. @. @. 0 @. 0 
\end{CD}
\]
and
\[
\begin{CD}
@. @. @. 0 @. 0 \\
@. @. @. @VVV @VVV \\
0 @>>> \mathcal{O}_X(-2)^3 @>>> \mathcal{O}_X(-1)^8 @>>> G @>>> \mathcal{O}_X(-1) @>>> 0 \\
@. @V=VV @V=VV @VVV @VVV \\
0 @>>> \mathcal{O}_X(-2)^3 @>>> \mathcal{O}_X(-1)^8 @>>> \mathcal{O}_X^6 @>>> \mathcal{O}_X(2) 
@>>> 0 \\
@. @. @. @VVV @VVV \\
@. @. @. \mathcal{O}_Z(2) @>=>> \mathcal{O}_Z(2) \\
@. @. @. @VVV @VVV \\
@. @. @. 0 @. 0 
\end{CD}
\]
These sheaves are not iterated extensions (\lq\lq non-commutative deformations'') 
of a collection $(\mathcal{O}_X(-2), \mathcal{O}_X(-1))$, but higher extensions (cf. \cite{multi}, \cite{ODP}).

(4) Martin Kalck noticed in a private communication that the above calculations 
on the $45$-dimensional algebra
match up at least numerically with a computation in the cluster theory and higher Auslander-Reiten sequences. 

\end{Rem} 


Graduate School of Mathematical Sciences, University of Tokyo,
Komaba, Meguro, Tokyo, 153-8914, Japan. 

kawamata@ms.u-tokyo.ac.jp

\end{document}